\documentclass[11pt]{article}
\usepackage{amsfonts,amssymb,amsbsy,amsmath,amscd,amstext,amsthm}
\usepackage{epsfig}

\theoremstyle{plain}
\newtheorem{lem}{Lemma}[section]
\newtheorem{thm}[lem]{Theorem}
\newtheorem{prop}[lem]{Proposition}
\newtheorem{coro}{Corollary}[lem]

\theoremstyle{definition}

\newtheorem{eg}{Example}

\theoremstyle{remark}
\newtheorem{rmk}[lem]{Remark}

\begin{document}

\title{The orbit of a $\beta$-transformation cannot lie in a small interval}

\author{DoYong\ Kwon}
\date{}
\maketitle

\begin{abstract}
For $\beta>1$, let $T_\beta:[0,1]\rightarrow [0,1)$ be the $\beta$-transformation. We consider an invariant $T_\beta$-orbit closure contained in a closed interval with diameter $1/\beta$, then define a function $\Xi(\alpha,\beta)$ by the supremum of such $T_\beta$-orbit with frequency $\alpha$ in base $\beta$, i.e., the maximum value in the $T_\beta$-orbit closure. This paper effectively determines the maximal domain of $\Xi$, and explicitly specifies all possible minimal intervals containing $T_\beta$-orbits.
\end{abstract}

\indent 2010 \textit{Mathematics Subject Classification:} 11A63, 37B10, 68R15.\\
\indent \textit{Keywords:} $\beta$-expansion, $\beta$-transformation, Sturmian word, Christoffel word. \newline

%
%
%
%
%
%
\section{Introduction}
In \cite{Ma}, Mahler considered hypothetical real number $\xi>0$, called a \textit{Z-number}, for which
$$0\leq \left\{ \xi\left(\frac{3}{2} \right)^n \right\}<\frac{1}{2}\ \ \mathrm{for\ every\ integer\ } n\geq0,$$
where $\{\cdot\}$ denotes the fractional part. This was motivated by a well-known connection between Waring's problem and the distribution of the fractional parts
$\left\{\left(\frac{3}{2} \right)^n \right\}$. See \cite{Fl}. Mahler proved that the set of Z-numbers is at most countable, though it is believed to be empty. Since then, Z-numbers have led us to study a more general problem \cite{FLP}. Given a number $\beta>1$ and an interval $[s,t] \subset [0,1]$, is there a real number $\xi>0$ for which $s\leq\{\xi \beta^n\}\leq t$ for every integer $n\geq0$? If so, what is the minimal diameter of $[s,t]$? Bugeaud and Dubickas gave a complete answer to the questions when $\beta\geq2$ is an integer \cite{BD}.

For $\beta>1$, the $\beta$-\textit{transformation} $T_\beta:[0,1]\rightarrow [0,1)$
is a map defined by
$$T_\beta(x)=\beta x \mod 1.$$
Let $\lfloor\cdot\rfloor$ and $\lceil\cdot\rceil$ be the usual floor and ceiling functions respectively. Symbolic dynamics of the $\beta$-transformations provides us with $\beta$-expansions \cite{Re,Pa,Bl}. For each $x\in[0,1]$, the $\beta$-\textit{expansion} of $x$ is a sequence of integers determined by the next iterated procedure:
$$d_\beta (x):=(x_i)_{i\geq1},\ \mathrm{where\ } x_i = \lfloor \beta T_\beta^{i-1} (x)\rfloor.$$
Note that $d_\beta (x)\in A_\beta^\mathbb{N}$ with $A_\beta:=\{0,1,\ldots,\lceil\beta\rceil-1\}$, and that the usual order in real numbers is compatible with  lexicographic order in $A_\beta^\mathbb{N}$. In other words,
$0\leq x<y\leq1$ if and only if $d_\beta (x)<d_\beta (y)$ lexicographically.
In studying the dynamics of the $\beta$-transformations, the lexicographic order via the $\beta$-expansions will play a crucial role in the whole paper. We say that a $\beta$-expansion of $x$ is \textit{finite} if it is of the form $d_\beta(x)=u0^\omega$, which is often written as $d_\beta(x)=u$. For $\beta>1$, we introduce a function $(\cdot)_\beta$, which sends each $a_1 a_2\cdots\in A_\beta^\mathbb{N}$ to a real number $\sum_{i=1}^\infty a_i/\beta^i$. In particular, we have $(d_\beta(x))_\beta=x$. If $u$ is a finite word, then $(u)_\beta$ is defined to be $(u0^\omega)_\beta$. For a finite word $v$, we write $v^\omega$ for $vv\cdots$.

Let $\alpha\geq0$ and $0\leq\rho\leq1$ be real. Then two functions
$s_{\alpha,\rho}, s'_{\alpha,\rho}:\mathbb{N}\rightarrow\mathbb{N}$ defined by
$$s_{\alpha,\rho}(n) := \lfloor \alpha(n+1) + \rho \rfloor -
\lfloor \alpha n +\rho \rfloor,$$
$$s'_{\alpha,\rho}(n) := \lceil
\alpha(n+1) + \rho \rceil - \lceil \alpha n +\rho \rceil,$$
yields infinite words $s_{\alpha,\rho}:=s_{\alpha,\rho}(0)s_{\alpha,\rho}(1)\cdots$ and $s'_{\alpha,\rho}:=s'_{\alpha,\rho}(0)s'_{\alpha,\rho}(1)\cdots$.
The word $s_{\alpha,\rho}$ (resp. $s'_{\alpha,\rho}$) is called a
\textit{lower} (resp. \textit{upper}) \textit{mechanical word} with \textit{slope} $\alpha$ and
\textit{intercept} $\rho$. One readily sees that $s_{\alpha,\rho}$ and $s'_{\alpha,\rho}$ are binary or unary words. Actually,
$$\mathrm{alph}(s_{\alpha,\rho}) =\mathrm{alph}(s'_{\alpha,\rho})=
\{\lceil\alpha\rceil-1, \lceil\alpha\rceil\},$$
except when $\alpha$ is an integer. For $\alpha\in\mathbb{N}$, $s_{\alpha,\rho}=s'_{\alpha,\rho}=\alpha^\omega$.

Now we can state the result of Bugeaud and Dubickas.

\begin{thm}[\cite{BD}]\label{BDirr}
Let $\beta\geq2$ be an integer and $\xi$ be an irrational number. Suppose that $s\leq\{\xi \beta^n\}\leq t$ for every integer $n\geq0$. Then $t-s$ cannot be smaller than $1/\beta$. Furthermore, $s\leq\{\xi \beta^n\}\leq s+1/\beta$ for every integer $n\geq0$ if and only if
$\xi=\lfloor\xi\rfloor+ (w)_\beta$ for some mechanical word $w$ with irrational slope.

\end{thm}

The present paper will generalize their work via $\beta$-transformations (Theorem \ref{MainMotiv}). We also obtain a corresponding result when $\xi$ is a rational number (Corollary \ref{RationalCase}).

\section{Definition of $\Xi$.}

Denoting
$$\langle t\rangle :=
\begin{cases}
    \{t\}, & \text{if\ $t\notin \mathbb{Z}$}, \\
    1, & \text{if\ $t\in \mathbb{Z}$},
  \end{cases}
$$
we define the $\overline{\beta}$-\textit{transformation} $\overline{T}_\beta:[0,1]\rightarrow (0,1]$ by
$$\overline{T}_\beta(x):=\langle \beta x\rangle.$$
Recall that the usual $\beta$-transformation $T_\beta$ is given by $T_\beta(x)=\{ \beta x\}$. It immediately follows that if $\beta\geq2$ is an integer and if $\xi\in[0,1]$ is an irrational number, then
$$T_\beta^n(\xi)=\overline{T}_\beta^n(\xi)= \{\xi \beta^n\}\  \mathrm{for\ every\ } n\geq0.$$
The $\overline{\beta}$-\textit{expansion} of $x\in[0,1]$ is a sequence of integers determined by the following rule:
$$\overline{d}_\beta (x)=(x_i)_{i\geq1},\ \mathrm{where\ } x_i = \lceil \beta \overline{T}_\beta^{i-1} (x)\rceil-1.$$
One observes that $\overline{\beta}$-expansions are almost the same as $\beta$-expansions in that if $d_\beta (x)$ is not finite then $d_\beta (x)=\overline{d}_\beta (x)$, and $T_\beta^n(x)=\overline{T}_\beta^n(x)$ for every $n\geq0$. But these $\overline{\beta}$-expansions and $\overline{\beta}$-transformations will make simple the statements of our main results below. Otherwise, we have to state them separately according as the $\beta$-expansions are finite or not. Two expansions $d_\beta (\cdot)$ and $\overline{d}_\beta (\cdot)$ correlate as the next proposition says. Note that the $\overline{\beta}$-expansions are never finite.

\begin{prop}\label{betabar}
\renewcommand{\theenumi}{\alph{enumi}}
\renewcommand{\labelenumi}{{\rm(\theenumi)}}
\begin{enumerate}
  \item $\overline{d}_\beta(1)$ is the left limit of $d_\beta$ at $x=1$, in other words,
  $$\overline{d}_\beta(1)= \lim_{x\to 1-} d_\beta(x)=
\begin{cases}
    d_\beta(1), & \text{if\ $d_\beta(1)$\ is\ not\ finite}, \\
    (\epsilon_1 \cdots \epsilon_{m-1}(\epsilon_m-1))^\omega, & \text{if\ $d_\beta(1)=\epsilon_1 \cdots \epsilon_m$,\ $\epsilon_m\neq0$}.
  \end{cases}
$$
\label{betabara}

  \item For any $x\in[0,1)$,
$$\overline{d}_\beta(x)=
\begin{cases}
    d_\beta(x), & \text{if\ $d_\beta(x)$\ is\ not\ finite}, \\
    (a_1 \cdots a_{m-1}(a_m-1))\overline{d}_\beta(1), & \text{if\ $d_\beta(x)=a_1 \cdots a_m$,\ $a_m\neq0$}.
  \end{cases}
$$
Here, we use the natural concatenation between a finite word $a_1 \cdots a_{m-1} (a_m-1)$ and an infinite word $\overline{d}_\beta(1)$. \label{betabarb}
\end{enumerate}

\end{prop}

\begin{proof}
\eqref{betabara} Suppose $d_\beta(1)=\epsilon_1 \cdots \epsilon_m$ with $\epsilon_m\neq0$. Then
$T_\beta^i(1)=\overline{T}_\beta^i(1)$ for $0\leq i\leq m-1$, and hence $\overline{d}_\beta(1)$ has a prefix $\epsilon_1 \cdots \epsilon_{m-1}$. But $\overline{T}_\beta^m(1)=1$ while $T_\beta^m(1)=0$, equivalently $\beta \overline{T}_\beta^{m-1} (1)$ is an integer. Whence
$\lceil \beta \overline{T}_\beta^{m-1} (1)\rceil-1 =\lfloor \beta \overline{T}_\beta^{m-1} (1)\rfloor-1 =\epsilon_m -1$. Now $\overline{T}_\beta^{m+1}(1)=\overline{T}_\beta(1)$, $\overline{T}_\beta^{m+2}(1)=\overline{T}_\beta^2(1), \ldots$ and so on.

\eqref{betabarb} Suppose $d_\beta(x)=a_1 \cdots a_m$ with $a_m\neq0$. A similar reasoning shows that $\overline{d}_\beta(x)$ begins with $a_1 \cdots a_{m-1}$, and that $\overline{T}_\beta^m(x)=1$.
\end{proof}

For $x\in[0,1]$, the \textit{frequency} of $x$ is defined to be the limit
$$\lim_{n\to\infty} \frac{1}{n}\sum_{k=1}^n a_k,$$
where $\overline{d}_\beta(x)=a_1 a_2 \cdots$. Roughly speaking, the frequency is the expectation value of digits in the $\overline{\beta}$-expansions.

\begin{eg}
Suppose that $\overline{d}_\beta(x)$ is a mechanical word of slope $\alpha$. Then the frequency of $x$ is equal to $\alpha$.
\end{eg}

For a nonempty subset $S\subset[0,1]$, we say that $S$ has \textit{frequency} $\alpha$ if, for every $x\in S$, its frequency has common value $\alpha$. If $\beta=2$, then $\overline{T}_2$ is nothing but the doubling map:
$x \mapsto 2x \mod 1$, and thus  the frequency of $S$ coincides with the rotation number studied in \cite{BS}. Consequently, the present paper may be a generalization of \cite{BS} in one possible direction.

We say that a set $S$ is invariant under a map $T$ if $T(S)=S$. If $S$ is, in particular, a $T$-orbit (closure), then the phrase `under a map $T$' will be often omitted.
Let $\alpha>0$ and $\beta>1$, and suppose that there is an invariant $\overline{T}_\beta$-orbit closure contained in $[t,t+\frac{1}{\beta}]\subset[0,1]$, and that it has frequency $\alpha$. Then we define a function $\Xi(\alpha,\beta)$ by the supremum of the $\overline{T}_\beta$-orbit, i.e., the maximum value in the $\overline{T}_\beta$-orbit closure. As for the well-definedness of $\Xi$, we will see below that at most one such $\overline{T}_\beta$-orbit closure exists (Proposition \ref{diam+rot}). Even the domain of $\Xi$ is nontrivial for the present, and hence to be determined later. It is worthwhile to mention that the current point of view extends the work of \cite{CK}, where invariant $T_\beta$-orbits contained only in $[1-\frac{1}{\beta},1]$ were investigated.

%
%
%
%
%
%

\section{Sturmian and Christoffel words.}

Let $\mathbb{N}$ be the set of nonnegative integers. Let $A$ be a finite alphabet. Throughout the paper, denote by $A^*$ (resp. $A^\mathbb{N}$)
the set of finite (resp. infinite) words over $A$. Then $A^*$ is a free monoid and the empty word $\varepsilon$ serves as its identity.
For a word $w \in
A^* \cup A^\mathbb{N}$, let $\mathrm{alph}(w)$ denote the set of letters appearing in $w$. The \textit{shift} $\sigma$ on $A^\mathbb{N}$ is a map defined by $\sigma(x_0 x_1 \cdots):=x_1 x_2 \cdots$. If an alphabet $A$ is a subset of $\mathbb{N}$, then the
lexicographic order on $A^\mathbb{N}$ is naturally extended to an order on $A^* \cup A^\mathbb{N}$ by embedding any $x \in A^*$ into
$x0^\omega\in (A\cup\{0\})^\mathbb{N}$. Let $w=a_1 a_2\cdots a_n \in A^*$. We write $\tilde{w}=a_n a_{n-1}\cdots a_1$ for the \textit{reversal} of $w$. A \textit{palindrome} is a finite word $u$ satisfying $\tilde{u}=u$.

For a letter $a\in A$, we mean by $|w|_a$ the number of occurrences of $a$ in $w$, whereas $|w|$ the number of letters, counting multiplicity, in $w$, i.e., $|w|=\sum_{a\in A}|w|_a$. We write $A^k:=\{w\in A^*:|w|=k\}$. A positive integer $p$ is called a \textit{period} of a finite word $w=a_1 a_2\cdots a_n$ provided $a_i=a_{i+p}$ for every $1\leq i\leq n-p$.

Suppose that $w \in A^*\cup A^\mathbb{N}$ is a binary words, say, $\mathrm{alph}(w)=\{a,b\}$. Then $w$ is said to be \textit{balanced} provided $||u|_b-|v|_b|\leq1$ (equivalently, $||u|_a-|v|_a|\leq1$) whenever both $u$ and $v$ are finite factors of $w$ with $|u|=|v|$.
Morse and Hedlund \cite{MH} showed that an infinite balanced word is either a mechanical word or a non-recurrent infinite word, called a \textit{skew word}, which is a mechanical word of rational slope after a nonempty prefix. This rational slope is also called the slope of the skew word. To be more precise, let $\{x,y\}=\{a,b\}$. Skew words are suffixes of
$\mu(x^l yx^\omega),\ l\in \mathbb{N}$ but not of $\mu(x^\omega)$, where $\mu$ is a composition of any finite number of morphisms $\psi_a$ and $\psi_b$ defined by
\begin{align*}
    \psi_a(a)=a,\quad & \psi_b(a)=ba,\\
    \psi_a(b)=ab,\quad & \psi_b(b)=b.
\end{align*}
See, e.g., \cite{AG2}. On the other hand, Coven and Hedlund \cite{CH} characterized the unbalancedness as follows.

\begin{prop}\label{unbalance}
Let $w \in A^*\cup A^\mathbb{N}$ and $\mathrm{alph}(w)=\{a,b\}$. Then $w$ is unbalanced if and only if there exists a palindrome $u$ such that both $aua$ and $bub$ are factors of $w$.
\end{prop}

If the slope $\alpha$ is irrational with $a+1=b=\lceil\alpha\rceil$, then $s_{\alpha,\rho}$ and $s'_{\alpha,\rho}$ are termed \textit{Sturmian words}, which are known to be simplest amongst all aperiodic infinite words. See, e.g., \cite{Lo}. Both $s_{\alpha,0}$ and $s'_{\alpha,0}$ have a common suffix $c_\alpha$ called the \textit{characteristic word}, that is,
\begin{equation}\label{Sturm-char}
s_{\alpha,0}=a c_\alpha,\ \ s'_{\alpha,0}=b c_\alpha.
\end{equation}

If $\alpha$ is rational, then one sees that $s_{\alpha,\rho}$ and $s'_{\alpha,\rho}$ are purely periodic for any $\rho$. Let $\alpha=p/q>0$ with $\gcd(p,q)=1$. Then the shortest period word $t_{p,q}$ (resp. $t'_{p,q}$) of $s_{\alpha,0}$ (resp. $s'_{\alpha,0}$) are called the \textit{lower} (resp. \textit{upper}) \textit{Christoffel words}. The lower and upper Christoffel words have a common factor $z_{p,q}$ called the \textit{central word}:
\begin{equation}\label{Christ-cent}
t_{p,q}=a z_{p,q}b,\ \ t'_{p,q}=bz_{p,q}a,
\end{equation}
where $a+1=b=\lceil\alpha\rceil$. One readily finds that the central word is a palindrome.

The case where intercept $\rho$ equals zero plays a special role in view of lexicographic order.

\begin{prop}\label{LexMech}
Let $\alpha>0$ be not an integer.
\renewcommand{\theenumi}{\alph{enumi}}
\renewcommand{\labelenumi}{{\rm(\theenumi)}}
\begin{enumerate}
  \item For any $\rho\in[0,1]$, the shift orbit closures of $s_{\alpha,\rho}$ are the set of lower and upper mechanical words with slope $\alpha$. The same statement is true for $s'_{\alpha,\rho}$.

  \item In particular, if $\alpha$ is rational, then for every $\rho\in[0,1]$ the shift orbits of $s_{\alpha,\rho}$ and $s'_{\alpha,\rho}$ coincide with that of $s_{\alpha,0}$.

  \item If $0<\rho<1$, then $s_{\alpha,0}\leq s_{\alpha,\rho}\leq s'_{\alpha,\rho}\leq s'_{\alpha,0}$. If $\alpha$ is irrational, then the first and the last inequalities are strict.
\end{enumerate}

\end{prop}

\begin{proof}
See \cite{Lo,BLRS}.
\end{proof}

For a finite word $w$, denote by $w^+$ the unique shortest palindrome having $w$ as a prefix. For instance,
$(abbab)^+=abbabba$. This palindrome $w^+$ is called the \textit{(right) palindromic closure} of $w$.
Over a binary alphabet $A=\{a,b\}$, a function $\mathrm{Pal}:A^*  \rightarrow A^*$ is defined as follows:
\renewcommand{\theenumi}{\roman{enumi}}
\renewcommand{\labelenumi}{{\rm(\theenumi)}}
\begin{enumerate}
  \item $\mathrm{Pal}(\varepsilon):=\varepsilon$,
  \item if $w=vz$ for some $z\in\{a,b\}$, then $\mathrm{Pal}(w):=(\mathrm{Pal}(v)z)^+$.
\end{enumerate}

In what follows, we need some combinatorics on central words. For more details, see \cite{Lo,BLRS} and the bibliography therein.

\begin{prop}\label{palin}
Let $P$ be the set of palindromes over $\{a,b\}$.
\renewcommand{\theenumi}{\alph{enumi}}
\renewcommand{\labelenumi}{{\rm(\theenumi)}}
\begin{enumerate}
 \item A word over $\{a,b\}$ is central if and only if it is a power of a single letter or belongs to $P\cap PabP =P\cap PbaP $.

 \item A word $u\in\{a,b\}^*$ is central if and only if $u=\mathrm{Pal}(v)$ for some $v\in\{a,b\}^*$.

 \item Given a central word $w$ being not a power of a single letter, there are unique palindromes $p$ and $q$ such that $w=pabq=qbap$. Furthermore, both $p$ and $q$ are also central, and $w$ has relatively prime periods $|p|+2$ and $|q|+2$.
\end{enumerate}
\end{prop}

Given a central word $u$, the next proposition enables us to find the word $v$ such that $u=\mathrm{Pal}(v)$.

\begin{prop}\label{FuncPal}
The function $\mathrm{Pal}:\{a,b\}^*  \rightarrow \{a,b\}^*$ is injective. To be more precise, let
$\varepsilon=p_1,p_2,\ldots,p_n,p_{n+1}=\mathrm{Pal}(w)$ be all the palindromic prefixes of a central word $\mathrm{Pal}(w)$ such that $0=|p_1|<|p_2|<\cdots <|p_n|<|\mathrm{Pal}(w)|$. Suppose further that all of
$p_1z_1,p_2z_2,\ldots,p_nz_n$ with $z_i\in\{a,b\}$ are prefixes of $\mathrm{Pal}(w)$, i.e., $z_i$ is the letter just after $p_i$. Then $w=z_1z_2\cdots z_n$.

\end{prop}

%
%
%
%
%
%

\section{Domain of $\Xi$.}\label{SecDomain}

The $\beta$-expansion of $1$ dominates the other ones of $x\in[0,1)$ in lexicographic order. And this property completely characterizes the possible $\beta$-expansions.

\begin{prop}[\cite{Pa}]\label{lexParry}
Given $\beta>1$, let $s\in A_\beta^\mathbb{N}$.
Then $s=d_\beta(x)$ for some $x\in[0,1)$ if
and only if
$$\sigma^n(s)< \overline{d}_\beta (1)\ \ \textit{for\ all\ } n\geq 0.$$
\end{prop}

We recall the devil's staircase in \cite{Kw-devil}. A function $\Delta:[0,\infty)\rightarrow \mathbb{R}$ is defined as follows. First, $\Delta(0):=1$. For $\alpha>0$, the value $\Delta(\alpha)$ is defined to be the real number $\beta$ for which the $\overline{\beta}$-expansion of $1$ is equal to $s'_{\alpha,0}$. In terms of the usual $\beta$-expansion, one may state

$$d_{\Delta(\alpha)}(1):=
\begin{cases}
    bc_\alpha=s'_{\alpha,0}, & \text{if\ $\alpha$\ is\ irrational}, \\
    bz_{p,q}b, & \text{if\ $\alpha=p/q$},
  \end{cases}
$$
where $b=\lceil\alpha\rceil$.

\begin{lem} \label{OrbMech}
Suppose that a $\overline{T}_\beta$-orbit of $x\in[0,1]$ is contained in $[t,t+\frac{1}{\beta}]\subset[0,1]$. Then the $\overline{\beta}$-expansion of $x$ is either a mechanical word or a skew word. In addition, if the $\overline{T}_\beta$-orbit closure of $x$ is invariant under $\overline{T}_\beta$, then the $\overline{\beta}$-expansion of $x$ is a mechanical word.
\end{lem}

\begin{proof}
Let $\overline{d}_\beta(t)=as$ for some $a\in A_\beta$ and $s\in A_\beta^\mathbb{N}$. Consequently, the $\overline{\beta}$-expansion of $t+\frac{1}{\beta}$ is given by $bs$, where $b=a+1$. If $\overline{d}_\beta(x)=r \in A_\beta^\mathbb{N}$, then one has $as \leq \sigma^n(r)\leq bs$ for every $n\geq0$, and therefore $\mathrm{alph}(r)=\{a,b\}$. If $r$ is unbalanced, then Proposition \ref{unbalance} guarantees a palindrome $u$ such that both $aua$ and $bub$ are factors of $r$. So $as\leq aua<bub\leq bs$ yields a contradiction, from which we conclude that $r$ is balanced, i.e., either a mechanical word or a skew word. If the $\overline{T}_\beta$-orbit closure of $x$ is invariant, then $r$ cannot be a skew word because it is non-recurrent.
\end{proof}

The next proposition prescribes the domain of $\Xi$.

\begin{prop} \label{diam+rot}
Let $\alpha>0$ and $\beta>1$.
\renewcommand{\theenumi}{\alph{enumi}}
\renewcommand{\labelenumi}{{\rm(\theenumi)}}
\begin{enumerate}
  \item If $\Delta(\alpha)\leq \beta$, then there exists a unique invariant $\overline{T}_\beta$-orbit closure such that it is contained in some $[t,t+\frac{1}{\beta}]\subset[0,1]$ and has frequency $\alpha$.

  \item If $\Delta(\alpha)> \beta$, then no invariant $\overline{T}_\beta$-orbit closure has $\alpha$ as its frequency.
\end{enumerate}

\end{prop}

\begin{proof}
Suppose $\Delta(\alpha)\leq \beta$, or equivalently $s'_{\alpha,0}\leq \overline{d}_\beta(1)$. Then Proposition \ref{LexMech} shows that the $\overline{T}_\beta$-orbit closure of $(s'_{\alpha,0})_\beta$ is invariant and contained in an interval
$[(s_{\alpha,0})_\beta,(s'_{\alpha,0})_\beta]$. Moreover, its frequency is equal to $\alpha$.
Noting \eqref{Sturm-char} and \eqref{Christ-cent} one finds $(s'_{\alpha,0})_\beta -(s_{\alpha,0})_\beta \leq1/\beta$ with equality if and only if $\alpha$ is irrational. If there are two such $\overline{T}_\beta$-orbit closures with frequency $\alpha$, then Lemma \ref{OrbMech}  associates them with mechanical words with the same slope $\alpha$. By Proposition \ref{LexMech}, the two $\overline{T}_\beta$-orbit closures should coincide.

If $\Delta(\alpha)> \beta$, then $s'_{\alpha,0}> \overline{d}_\beta(1)$. Hence Lemma \ref{OrbMech} tells us that some point in such $\overline{T}_\beta$-orbit necessarily violates Proposition \ref{lexParry}.
\end{proof}

\begin{rmk}
In the previous result, the invariance is essential for the uniqueness. Let $0\leq a = b-1$ be integers and $\beta\geq b+1$ be a real number. Suppose $\overline{d}_\beta(t)=(ab)^\omega$. Then the interval $[t,t+\frac{1}{\beta}]$ contains both $\overline{T}_\beta$-orbits of $((ab)^\omega)_\beta$ and $(b(ba)^\omega)_\beta$, which have a common frequency $a+\frac{1}{2}$. One may verify that $b(ba)^\omega$ is a skew word.
\end{rmk}

We recall that the function $\Xi(\alpha,\beta)$ is defined by the supremum of the $\overline{T}_\beta$-orbit contained in some $[t,t+\frac{1}{\beta}]\subset[0,1]$ and having frequency $\alpha$. Proposition \ref{diam+rot} now tells us that for $\alpha$ and $\beta$ satisfying $\Delta(\alpha)\leq \beta$, the function $\Xi(\alpha,\beta)$ is well-defined,
and that if $\Delta(\alpha)> \beta$ then $\Xi(\alpha,\beta)$ is not defined at all. A more detailed description of the maximal domain of $\Xi$ will be presented in Theorem \ref{DomainCond}.

Owing to Lemma \ref{OrbMech}, we can alternatively define $\Xi$ by
$$\Xi(\alpha,\beta)=(s'_{\alpha,0})_\beta.$$
Apparently, this is more tangible in many cases.

Let $\beta>1$ be fixed. Suppose $\Delta(\alpha)\leq \beta$. If $\alpha$ is irrational, then the invariant $\overline{T}_\beta$-orbit closure of $(s'_{\alpha,0})_\beta$ is contained in $[t,t+\frac{1}{\beta}]$ for a unique $t=(s_{\alpha,0})_\beta$.
On the other hand, if $\alpha$ is rational, then that orbit closure is contained in $[t,t+\frac{1}{\beta}]$ for $(s'_{\alpha,0})_\beta- \frac{1}{\beta}\leq t \leq (s_{\alpha,0})_\beta$.

For the other direction, for a fixed $\beta>1$, let an interval $[t,t+\frac{1}{\beta}]\subset [0,1]$ be given. Then we will see next that the interval actually contains a unique invariant $\overline{T}_\beta$-orbit closure. We prove uniqueness first.

\begin{prop}\label{uniTb}
Let $\beta>1$ be fixed, and $I=[t,t+\frac{1}{\beta}]\subset [0,1]$ be given. Then there exists
at most one invariant $\overline{T}_\beta$-orbit closure contained in $I$.

\end{prop}

\begin{proof}
Suppose that $I$ contains two different invariant $\overline{T}_\beta$-orbit closures. By Lemma \ref{OrbMech}, they are orbit closures of some points whose $\overline{\beta}$-expansions are mechanical words. If the mechanical words have the same slope, then Proposition \ref{LexMech} shows that the two orbit closures coincide. So assume that $I$ contains $\overline{T}_\beta$-orbit closures of $(s'_{\alpha_1,0})_\beta$ and $(s'_{\alpha_2,0})_\beta$ for some $0<\alpha_1<\alpha_2$.

We claim that $(s'_{\alpha_2,0})_\beta -(s_{\alpha_1,0})_\beta >1/\beta$, which implies that not both
$\overline{T}_\beta$-orbit closures of $(s'_{\alpha_1,0})_\beta$ and $(s'_{\alpha_2,0})_\beta$  can be contained in $I$.
Indeed, if $\alpha_2$ is irrational, then $(s'_{\alpha_2,0})_\beta -(s_{\alpha_2,0})_\beta = 1/\beta$. Since $s_{\alpha_2,0}>s_{\alpha_1,0}$ and $s_{\alpha_2,0}$ is aperiodic, one has $(s_{\alpha_2,0})_\beta >(s_{\alpha_1,0})_\beta$. Suppose that $\alpha_2=p/q$ is rational, and so $s'_{\alpha_2,0}=(bz_{p,q}a)^\omega$, where $a+1=b=\lceil\alpha_2\rceil$. Then the $\overline{\beta}$-expansion of $(s'_{\alpha_2,0})_\beta-1/\beta$ is equal to  $az_{p,q}a(bz_{p,q}a)^\omega$.
Now it suffices to show that $az_{p,q}a(bz_{p,q}a)^\omega > s_{\alpha_1,0}$. On the contrary, we assume that $az_{p,q}a(bz_{p,q}a)^\omega < s_{\alpha_1,0}$. Note that the equality never holds. From $\alpha_1<\alpha_2$ it follows that
$$az_{p,q}a(bz_{p,q}a)^\omega < s_{\alpha_1,0}<az_{p,q}b(az_{p,q}b)^\omega=s_{\alpha_2,0}.$$
One finds then that $s_{\alpha_1,0}$ begins with $az_{p,q}a$, say, $s_{\alpha_1,0}=az_{p,q}a s$ for some $s\in\{a,b\}^\mathbb{N}$. Therefore, we have
$$s'_{\alpha_2,0}=(bz_{p,q}a)^\omega < s=s_{\alpha_1,\{\alpha_1 q\}} <s'_{\alpha_1,0},$$
which contradicts $\alpha_1<\alpha_2$.
\end{proof}

The rest of this section is devoted to finding the invariant $\overline{T}_\beta$-orbit closure contained in a given interval $[t,t+\frac{1}{\beta}]\subset [0,1]$. We also compute its frequency explicitly and effectively. To this end, we need a closer look at Christoffel words. This view is a slight generalization of a recent work done by Allouche and Glen \cite{AG} --- from binary alphabet to finite alphabet. As a byproduct, we obtain a considerably simpler proof of \cite[Lemma 27]{AG}.

\begin{lem}[\cite{AG}]\label{CentPref}
Suppose that $w$ is a central word with $\mathrm{alph}(w)=\{a,b\}$. Let $p$ and $q$ be the
unique pair of central words satisfying $w=pabq=qbap$. Then $wabq$ (resp. $wbap$) is a prefix of $(qba)^\omega$ (resp. $(pab)^\omega$).
\end{lem}

\begin{proof}
We prove only the assertion for $wabq$. The case of $wbap$ is a symmetric argument.
One derives $wabq=qbapabq=(qba)^2 p$. Therefore, it suffices to show that $p$ is a prefix of $(qba)^\omega$.
If $|p|\leq|qba|$, then it is obvious. If $|p|>|qba|$, then the claim follows from the fact that $w=pabq=qbap$ has a period $|q|+2$.
\end{proof}

Given an interval $[t,t+\frac{1}{\beta}]\subset [0,1]$, the following two theorems determine the unique invariant $\overline{T}_\beta$-orbit closure contained in it, and its frequency. Thanks to the uniqueness in Proposition \ref{uniTb}, all we have to prove is to check whether the given orbit is actually contained in $[t,t+\frac{1}{\beta}]$.

\begin{thm}\label{MeagerCase}
Given $\beta>1$, let $0\leq t \leq 1-\frac{1}{\beta}$ and $\overline{d}_\beta(t)=as$ with a nonnegative integer $a=:b-1$. The interval $I:=[t,t+\frac{1}{\beta}]$ contains a unique invariant $\overline{T}_\beta$-orbit closure as follows.
\renewcommand{\theenumi}{\alph{enumi}}
\renewcommand{\labelenumi}{{\rm(\theenumi)}}
\begin{enumerate}
  \item If $s\leq a^\omega$, then $I$ contains the orbit closure of $(a^\omega)_\beta$, whose frequency is $a$.\label{MeagerCase1}
  \item If $s\geq b^\omega$, then $I$ contains the orbit closure of $(b^\omega)_\beta$, whose frequency is $b$.\label{MeagerCase2}
  \item If $as=s_{\alpha,0}$ or $bs=s'_{\alpha,0}$ for some $\alpha\in[a,b]$, then $I$ contains the orbit closure of $(s'_{\alpha,0})_\beta$. So the frequency is $\alpha$.\label{MeagerCase3}
\end{enumerate}

\end{thm}

\begin{proof}
\eqref{MeagerCase1} and \eqref{MeagerCase2} are easy to check. \eqref{MeagerCase3} is a consequence of Proposition \ref{LexMech}.
\end{proof}

The other cases excluded in Theorem \ref{MeagerCase} are captured by the next one. The proof of the corresponding theorem in \cite{AG} also works here with minor but careful modifications.

\begin{thm}\label{GenericCase}
In addition to the assumption in Theorem \ref{MeagerCase}, suppose that $s$ satisfies none of the above three cases. Let $u$ be the longest central prefix of $s$ over $\{a,b\}$. The interval $I:=[t,t+\frac{1}{\beta}]$ contains a unique invariant $\overline{T}_\beta$-orbit closure as follows.
\renewcommand{\theenumi}{\alph{enumi}}
\renewcommand{\theenumii}{\roman{enumii}}
\renewcommand{\labelenumi}{{\rm(\theenumi)}}
\renewcommand{\labelenumii}{{\rm(\theenumii)}}
\begin{enumerate}
  \item Suppose $u=a^k$ for some $k\geq1$, and let $s=a^k vs'$ where $v\in A_\beta^{k+1}$ and $s'\in A_\beta^\mathbb{N}$. \label{GenericCase1}

  \begin{enumerate}
    \item If $v<a^{k+1}$, then $s$ falls into Case (\ref{MeagerCase1}) of Theorem \ref{MeagerCase}. \label{GenericCase1-a}
    \item If $a^{k+1}<v<ba^k$, then $I$ contains the orbit closure of $((ba^{k+1})^\omega)_\beta$, whose frequency is $a+\frac{1}{k+2}$. \label{GenericCase1-b}
    \item If $v>ba^k$, then $I$ contains the orbit closure of $((ba^k)^\omega)_\beta$, whose frequency is $a+\frac{1}{k+1}$. \label{GenericCase1-c}
  \end{enumerate}

  \item Suppose $u=b^k$ for some $k\geq1$, and let $s=b^k vs'$ where $v\in A_\beta^{k+1}$ and $s'\in A_\beta^\mathbb{N}$. \label{GenericCase2}
  \begin{enumerate}
    \item If $v>b^{k+1}$, then $s$ falls into Case (\ref{MeagerCase2}) of Theorem \ref{MeagerCase}. \label{GenericCase2-a}
    \item If $ab^k<v<b^{k+1}$, then $I$ contains the orbit closure of $((b^{k+1}a)^\omega)_\beta$, whose frequency is $b-\frac{1}{k+2}$. \label{GenericCase2-b}
    \item If $v<ab^k$, then $I$ contains the orbit closure of $((b^k a)^\omega)_\beta$, whose frequency is $b-\frac{1}{k+1}$. \label{GenericCase2-c}
  \end{enumerate}

  \item Suppose that $p$ and $q$ are the
unique pair of central words satisfying $u=pabq=qbap$. Let $s=uxyvs'$ where $x,y\in A_\beta$, $v\in A_\beta^{|u|+2}$ and $s'\in A_\beta^\mathbb{N}$. \label{GenericCase3}
  \begin{enumerate}
    \item Either if $xy=ab$ and $v>uab$, or if $xy=ba$ and $v<uba$, then $I$ contains the orbit closure of $((bua)^\omega)_\beta$, whose frequency is $a+\frac{|u|_b+1}{|u|+2}$. \label{GenericCase3-a}
    \item Either if $xy=ab$ and $v<uab$, or if $xy\leq aa$, then $I$ contains the orbit closure of $((bqa)^\omega)_\beta$, whose frequency is $a+\frac{|q|_b+1}{|q|+2}$. \label{GenericCase3-b}
    \item Either if $xy=ba$ and $v>uba$, or if $xy\geq bb$, then $I$ contains the orbit closure of $((bpa)^\omega)_\beta$, whose frequency is $a+\frac{|p|_b+1}{|p|+2}$. \label{GenericCase3-c}
  \end{enumerate}
\end{enumerate}

\end{thm}

In every inequality above where $v$ is involved, equality never holds because $u$ is the longest central prefix of $s$ over $\{a,b\}$.

\begin{proof}
\eqref{GenericCase2} is symmetric to \eqref{GenericCase1}.

\paragraph{(\ref{GenericCase1-b}) $\mathbf{a^{k+1}<v<ba^k}$.} One sees
$$as=a^{k+1}vs'<(a^{k+1}b)^\omega$$
since $v<ba^k$. The first letter of $v$ is either $a$ or $b$. But only $b$ is possible --- otherwise, there is a longer central prefix of $s$ over $\{a,b\}$ than $u$. Accordingly,
$$bs=ba^kvs'>(ba^{k+1})^\omega.$$

\paragraph{(\ref{GenericCase1-c}) $\mathbf{v>ba^k}$.} We note
$$as=a^{k+1}vs'<(a^kb)^\omega,\ \mathrm{and\ } (ba^k)^\omega<ba^kvs'=bs.$$

\paragraph{(\ref{GenericCase3-a}) $\mathbf{xy=ab, v>uab}$, or $\mathbf{xy=ba, v<uba}$.}
Let $xy=ab$ and $v>uab$. Then
$$as=auabvs'<(aub)^\omega,\ \mathrm{and\ }
(bua)^\omega=buabuab(uab)^\omega<buabvs'.$$
The other case is symmetric.

\paragraph{(\ref{GenericCase3-b}) $\mathbf{xy=ab, v<uab}$, or $\mathbf{xy\leq aa}$.}
In either case,
$$(bqa)^\omega<bqbap=bu<buxyvs'=bs.$$
Let $xy=ab$ and $v<uab$. We claim $v<q$. Since $v<uab$ and $|v|>|q|$, all we have to do is to exclude the case where $q$ is a prefix of $v$. But by Proposition \ref{palin}, $uabq$ is central, which contradicts that $u$ is the longest. One thus finds
$$as=auabvs'<auabq<a(qba)^\omega=(aqb)^\omega,$$
where Lemma \ref{CentPref} is exploited in the last inequality.

If $xy\leq aa$, then
$$as=auxyvs'<auab<a(qba)^\omega=(aqb)^\omega.$$

\paragraph{(\ref{GenericCase3-c})} This is symmetric to \eqref{GenericCase3-b}.
\end{proof}

Let $\beta>1$ be not an integer and $\overline{d}_\beta(1)=bs$ with $b=\lfloor\beta\rfloor$. We define $\mathrm{Freq}(\beta)$, depending on $s$, by the frequencies that appear in Theorem \ref{MeagerCase} and \ref{GenericCase}. For an integer $\beta>1$, define $\mathrm{Freq}(\beta):=\beta-1$. This value is effectively computable from the infinite word $s$. Theorem \ref{MeagerCase} and \ref{GenericCase} give us the procedure to find $\mathrm{Freq}(\beta)$.

\begin{eg}
\begin{itemize}
  \item Let $\beta_1=\frac{1+\sqrt{5}}{2}$, and $\beta_2$ be the smallest Pisot number, i.e., the dominant real zero of $x^3-x-1$. Then
      $$\mathrm{Freq}(\beta_1)=\frac{1}{2}\ \ \mathrm{and\ }\ \mathrm{Freq}(\beta_2)=\frac{1}{5},$$
      because $\overline{d}_{\beta_1}(1)=(10)^\omega$ and $\overline{d}_{\beta_2}(1)=(10000)^\omega$. Both cases fall into Theorem \ref{MeagerCase} \eqref{MeagerCase3}.

  \item A direct computation shows $\overline{d}_{\pi}(1)=3 0 1 1 0 2 1 1 1 0\cdots$. Accordingly, this case falls into Theorem \ref{MeagerCase} \eqref{MeagerCase1} with $a=b-1=2$. Therefore,
      $$\mathrm{Freq}(\pi)=2.$$

  \item One has $\overline{d}_{\sqrt{7}}(1)=2 1 1 2 0 2 0 1 0 2\cdots=:bs$
      with $b=2$ and $s\in \{0,1,2\}^\mathbb{N}$. Clearly, $11$ is the longest central prefix of $s$ over $\{1,2\}$. Now we appeal to Theorem \ref{GenericCase} \eqref{GenericCase1-b}.
      Since $111<202<211$, we conclude
      $$\mathrm{Freq}(\sqrt{7})=\frac{5}{4}.$$

  \item In $\overline{d}_{3/2}(1)=1 0 1 0 0 0 0 0 1 0 0 1 0\cdots=:bs$, first note that $\overline{d}_{3/2}(1)$ is not balanced since both $101$ and $000$ appear.
      We apply Proposition \ref{FuncPal} to find the longest central prefix of $s$.
      Let $z_1=0$. Then $\mathrm{Pal}(z_1)=0$ is a prefix of $s$. Let $z_2=1$. Then $\mathrm{Pal}(z_1z_2)=010$ is also a prefix of $s$. Let $z_3=0$.
      Then $\mathrm{Pal}(z_1z_2z_3)=010010$ is no more a prefix of $s$. Hence $u=010$ is the longest central prefix of $s$. In this case, Theorem \ref{GenericCase} \eqref{GenericCase3-b} works with $p=\varepsilon, q=0, x=0, y=0$, and $v=00100$. Since $xy\leq 00$, we have
      $$\mathrm{Freq}(\frac{3}{2})=\frac{1}{3}.$$
\end{itemize}
\end{eg}

The maximal domain of $\Xi$ can be expressed in terms of $\mathrm{Freq}(\beta)$.

\begin{thm}\label{DomainCond}
\renewcommand{\theenumi}{\alph{enumi}}
\renewcommand{\labelenumi}{{\rm(\theenumi)}}
\begin{enumerate}
  \item For a fixed $\beta>1$, the function $\Xi(\alpha,\beta)$ is well defined for every $0<\alpha\leq \mathrm{Freq}(\beta)$. The bound $\mathrm{Freq}(\beta)$ is best possible. \label{DomainCond-1}

  \item For a fixed $\alpha>0$, the function $\Xi(\alpha,\beta)$ is well defined for every $\beta\geq \Delta(\alpha)$. The bound $\Delta(\alpha)$ is best possible. \label{DomainCond-2}

  \item In particular, if $\overline{d}_\beta(1)$ is a mechanical word of slope $\alpha_0$, then $\Xi(\alpha,\beta)$ is well defined as long as $0< \alpha\leq \alpha_0$. \label{DomainCond-3}
\end{enumerate}

\end{thm}

\begin{proof}
\eqref{DomainCond-2} was proved in Proposition \ref{diam+rot}, and \eqref{DomainCond-3} is a special case of \eqref{DomainCond-1}. We prove \eqref{DomainCond-1}. Let $0<\alpha\leq \mathrm{Freq}(\beta)=:\alpha_0$.
Then the $\overline{T}_\beta$-orbit closure of $(s'_{\alpha_0,0})_\beta$ is contained in the interval $[1-\frac{1}{\beta},1]$. Since $s'_{\alpha,0} \leq s'_{\alpha_0,0} \leq \overline{d}_\beta(1)$, Proposition \ref{LexMech} guarantees that the $\overline{T}_\beta$-orbit closure of $(s'_{\alpha,0})_\beta$ is invariant and contained in some $[t,t+\frac{1}{\beta}]\subset [0,1]$.

Suppose that $\alpha> \mathrm{Freq}(\beta)=:\alpha_0>0$ and that $\Xi(\alpha,\beta)$ is defined. Equivalently, we can say, owing to Lemma \ref{OrbMech}, that  the $\overline{T}_\beta$-orbit closure of $(s'_{\alpha,0})_\beta$ is contained in some $[t,t+\frac{1}{\beta}]\subset [0,1]$. On the other hand, the $\overline{T}_\beta$-orbit closure of $(s'_{\alpha_0,0})_\beta$ is contained in the interval $[1-\frac{1}{\beta},1]$. From $s'_{\alpha_0,0} < s'_{\alpha,0}$ it follows that the $\overline{T}_\beta$-orbit closure of $(s'_{\alpha,0})_\beta$ is also contained in $[1-\frac{1}{\beta},1]$, which contradicts the uniqueness in Proposition \ref{uniTb}.
\end{proof}

Other properties of the function $\Xi$ may be pursued further in subsequent works. See, e.g., \cite{Kw-twosing} for its analytical point of view.

\section{Diameters of $\overline{T}_\beta$-orbits.}

In this section, we return to the original motivation.

Let $\beta>1$. For $\xi\in[0,1]$, we set
$$\mathrm{Diam}_\beta(\xi):= \sup_{n\geq0}\overline{T}_\beta^n(\xi) -\inf_{n\geq0}\overline{T}_\beta^n(\xi).$$

\begin{thm} \label{MainMotiv}
Let $\beta>1$ and $\xi\in[0,1]$. Suppose $\mathrm{Diam}_\beta(\xi)\leq 1/\beta$. Then $\xi$ satisfies one of the following.
\renewcommand{\theenumi}{\alph{enumi}}
\renewcommand{\labelenumi}{{\rm(\theenumi)}}
\begin{enumerate}
  \item $\overline{d}_\beta(\xi)$ is a Sturmian word of slope less than or equal to $\mathrm{Freq}(\beta)$, and $\mathrm{Diam}_\beta(\xi)=1/\beta$. \label{MainMotiv-1}

  \item For some rational $p/q\leq \mathrm{Freq}(\beta)$, $\overline{d}_\beta(\xi)$ is a mechanical word of slope $p/q$, and $\mathrm{Diam}_\beta(\xi)=\frac{\beta^{q-1}-1}{\beta^q-1}$. \label{MainMotiv-2}

  \item For some rational $p/q\leq \mathrm{Freq}(\beta)$, $\overline{d}_\beta(\xi)$ is a skew word of slope $p/q$, and $\mathrm{Diam}_\beta(\overline{T}_\beta^k(\xi))=\frac{\beta^{q-1}-1}{\beta^q-1}$ for all $k$ sufficiently large. \label{MainMotiv-3}
\end{enumerate}

\end{thm}

\begin{proof}
If $\mathrm{Diam}_\beta(\xi)\leq 1/\beta$, then Lemma \ref{OrbMech} shows that the $\overline{\beta}$-expansion of $\xi$ is either a mechanical word or a skew word of some slope less than or equal to $\mathrm{Freq}(\beta)$. In the case where $\overline{d}_\beta(\xi)$ is a Sturmian word of an irrational slope $\alpha$ with $a+1=b=\lceil\alpha\rceil$, Proposition \ref{LexMech} tells us that $$\sup_{n\geq0}\overline{T}_\beta^n(\xi)=(s'_{\alpha,0})_\beta=(bc_\alpha)_\beta\ \ \mathrm{and\ \ }
\inf_{n\geq0}\overline{T}_\beta^n(\xi)=(s_{\alpha,0})_\beta=(ac_\alpha)_\beta.$$
So, $\mathrm{Diam}_\beta(\xi)=1/\beta$ follows.
When $\overline{d}_\beta(\xi)$ is a mechanical word of a rational slope $p/q$, we have
\begin{align*}
   \mathrm{Diam}_\beta(\xi) &=((bz_{p,q}a)^\omega)_\beta -((az_{p,q}b)^\omega)_\beta\\
   &= \left(\frac{1}{\beta}-\frac{1}{\beta^q} \right)
   \left(1+\frac{1}{\beta^q}+\frac{1}{\beta^{2q}}+\cdots \right)
   =\frac{\beta^{q-1}-1}{\beta^q-1}.
\end{align*}
If $\overline{d}_\beta(\xi)$ is a skew word of slope $p/q$ with preperiod $m$, then the $\overline{\beta}$-expansion of $\overline{T}_\beta^k(\xi)$ is a mechanical word of the same slope for every $k\geq m$.

\end{proof}

Now Theorem \ref{BDirr} is extended to the case where $\xi$ is a rational number.

\begin{coro} \label{RationalCase}
Let $\beta\geq2$ be an integer and $\xi$ be a rational number. Suppose $\mathrm{Diam}_\beta(\xi)\leq 1/\beta$. Then $\xi$ satisfies one of the following.
\renewcommand{\theenumi}{\alph{enumi}}
\renewcommand{\labelenumi}{{\rm(\theenumi)}}
\begin{enumerate}
  \item The $\overline{\beta}$-expansion of $\{\xi\}$ is a mechanical word $(e_1 \cdots e_q)^\omega$ of some rational slope $p/q\leq \beta-1$, and
      $$\xi= \lfloor\xi\rfloor+ \frac{e_1 \beta^{q-1}+e_2 \beta^{q-2}+\cdots +e_q}{\beta^q -1}.$$

  \item The $\overline{\beta}$-expansion of $\{\xi\}$ is a skew word $e_1 \cdots e_m(e_{m+1}\cdots e_{m+q})^\omega$ of some rational slope $p/q\leq \beta-1$, and
      $$\xi= \lfloor\xi\rfloor+ \frac{e_1 \beta^{m+q-1}+e_2 \beta^{m+q-2}+\cdots +e_{m+q}
                                      -(e_1 \beta^{m-1}+e_2 \beta^{m-2}+\cdots +e_m)}
      {\beta^m(\beta^q -1)}.$$

\end{enumerate}
\end{coro}

\begin{proof}
All we have to do is to exclude the case \eqref{MainMotiv-1} of Theorem \ref{MainMotiv}.
Suppose that $\overline{d}_\beta(\xi)$ is a Sturmian word $s$, in other words,
$\xi=(s)_\beta$. Then any of \cite{FM} or \cite{AB} proves that $\xi$ is a transcendental number.
\end{proof}

\bigskip
%
%
%
%
%
%

{\bf Acknowledgments.}
\medskip

This research was supported by Basic Science Research Program through the National Research Foundation of Korea(NRF) funded by the Ministry of Education, Science and Technology (2011-0004055).

%
%
%
%
%
%

\noindent Department of Mathematics,\\
Chonnam National University,\\
Gwangju 61186, Republic of Korea\\
E-mail: \textsf{doyong@jnu.ac.kr}

\end{document}


\title{Addendum to ``The orbit of a $\beta$-transformation cannot lie in a small interval, J. Korean Math. Soc. 49 (2012), no. 4, 867-879"}

\author{DoYong\ Kwon}
\date{}
\maketitle

\begin{abstract}
The paper mentioned in the title is completed by this Addendum.
\end{abstract}

\indent 2010 \textit{Mathematics Subject Classification:} 11A63, 37B10, 68R15.\\
\indent \textit{Keywords:} $\beta$-expansion, $\beta$-transformation, Sturmian word, Christoffel word. \newline

Let $A=\{a_1,a_2,\ldots,a_n\}$ be a fixed real alphabet with $a_1 <a_2<\ldots<a_n$, and $\beta>1$ a real number. We say that an infinite word $(x_i)_{i=1}^\infty$ over $A$ is an \textit{expansion} of $x$ in \textit{base} $\beta$ if
$$x =\sum_{i=1}^\infty \frac{x_i}{\beta^i}$$
holds.
If $\beta$ is not an integer, then the expansions are no more unique in general.
The uniqueness of expansions changes drastically as the base $\beta$ varies.
It turned out that there is a kind of boundary between uniqueness and non-uniqueness of expansions.
At first, one readily notes that both $a_1^\infty:=a_1 a_1 \cdots$ and  $a_n^\infty$ are unique expansions in any base $\beta>1$. It was proved in \cite{KLP} that there exists a real number $\beta_0>1$ satisfying the following:
\renewcommand{\theenumi}{\roman{enumi}}
\renewcommand{\labelenumi}{{\rm(\theenumi)}}
\begin{enumerate}
 \item if $1<\beta<\beta_0$, then every $x\in\left( \frac{a_1}{\beta-1}, \frac{a_n}{\beta-1} \right)$ has multiple expansions in base $\beta$,
 \item if $\beta>\beta_0$, then some $x\in\left( \frac{a_1}{\beta-1}, \frac{a_n}{\beta-1} \right)$ has a unique expansion in base $\beta$.
\end{enumerate}
The number $\beta_0$ is called the \textit{critical base} or \textit{generalized golden ratio} for $A$.
The current state of the art is very far from a complete characterization of critical bases for general alphabets $A$. Recently, the present author recognized that the work in \cite{Kw-ordered} may play a key role when the alphabet $A$ consists of three letters. He gave in \cite{Kw-golden3} a procedure to determine the critical bases for ternary alphabets. An effective and algebraic algorithm was possible thanks to Theorem 4.7 and 4.8 of \cite{Kw-ordered}. A computer program implementing the main idea was presented. And this program revealed a missing case in the theorems.
To the previous twelve cases in Theorem 4.7 and 4.8, must we add another new case at Theorem 4.8 (c) (iv).
The rest of this Addendum will make complete the paper mentioned in the title.
\bigskip

\noindent {\bf Theorem 4.8}\ \  $\cdots$\\
\indent \quad \quad \quad $\vdots$\\
\indent (c)\ \  $\cdots$\\
\indent \quad \quad \quad $\vdots$\\
\indent \quad (iv) {\it If $ab<xy<ba$, then $I$ contains the orbit closure of $((bua)^\omega)_\beta$, whose frequency is $a+\frac{|u|_b+1}{|u|+2}$.}

\begin{proof}
We have
$$as=auxyvs' < (aub)^\omega <(bua)^\omega < buxyvs' =bs.$$
\end{proof}

%
%
%
%
%
%

\noindent Department of Mathematics,\\
Chonnam National University,\\
Gwangju 61186, Republic of Korea\\
E-mail: \textsf{doyong@jnu.ac.kr}